\documentclass[11pt,reqno]{amsart}
\usepackage{amsmath,amsthm,amscd,amsfonts,amssymb,color}
\usepackage{cite}
\usepackage[mathscr]{eucal}

\usepackage[bookmarksnumbered,colorlinks,plainpages]{hyperref}
\newtheorem{theorem}{Theorem}[section]

\newtheorem*{Acknowledgement}{\textnormal{\textbf{Acknowledgement}}}

\newtheorem{corollary}[theorem]{Corollary}
\theoremstyle{definition}
\newtheorem{definition}[theorem]{Definition}
\newtheorem{example}[theorem]{Example}
\newtheorem{Open Prob}[theorem]{Open Problem}
\theoremstyle{remark}

\numberwithin{equation}{section}
\def\DJ{\leavevmode\setbox0=\hbox{D}\kern0pt\rlap{\kern.04em\raise.188\ht0\hbox{-}}D}
\begin{document}

\title[Some remarks on the Metrizability of $\mathcal{F}$-metric spaces]{Some remarks on the metrizability of $\mathcal{F}$-metric spaces}

\author[S.\ Som, A.\ Bera, L.K.\ Dey,]
{Sumit Som$^{1}$, Ashis Bera$^{2}$,  Lakshmi Kanta Dey$^{3}$}

\address{{$^{1}$} Sumit Som,
                    Department of Mathematics,
                    National Institute of Technology
                    Durgapur, India.}
                    \email{somkakdwip@gmail.com}
\address{{$^{2}$} Ashis Bera
                    Department of Mathematics,
                    National Institute of Technology
                    Durgapur, India.}
                    \email{beraashis.math@gmail.com}
\address{{$^{3}$} Lakshmi Kanta Dey,
                    Department of Mathematics,
                    National Institute of Technology
                    Durgapur, India.}
                    \email{lakshmikdey@yahoo.co.in}


\begin{abstract}
In this manuscript, we claim that the newly introduced $\mathcal{F}$-metric space \cite[\, M.~Jleli and B.~Samet, On a new generalization of metric spaces, J. Fixed Point Theory Appl, 20(3) 2018]{JS1} is metrizable. Also, we deduce that the notions of convergence, Cauchy sequence, completeness due to Jleli and Samet  for $\mathcal{F}$-metric spaces are equivalent with that of usual metric spaces. Moreover, we assert that the Banach contraction principle in the context of $\mathcal{F}$-metric spaces is a direct consequence of its standard metric counterpart.
\end{abstract}

\maketitle

\setcounter{page}{1}

\centerline{}

\centerline{}

\section{\bf Introduction}
\baselineskip .55 cm
Many mathematicians are attracted to work on a topic which is more fundamental and has a lot of applications in many diversified fields. One of the major motivations is to generalize or weaken a certain structure and develop new results compatible to the weaker one. Indeed, there are generalizations which  genuinely develop the subject as a whole and also, there are some which contribute nothing new to the literature.
Likewise, to generalize the notion of distance functions, in 1906, Fr$\acute{e}$chet first introduced the concept of metric spaces and a century later, we have numerous generalizations of the metric structure. Intent readers are referred to \cite{D3,M2,HS} and references therein for some relevant extensions. Unfortunately, some of the generalizations become redundant and turn into metrizable merely adding premise to the subject. As in 2007, Huang and Zhang \cite{HZ} introduced the notion of a cone metric space on a positive cone in a Banach space. Following that, a lot of research articles dealt with the setting and evolved the structure with a number of results. Although in 2011, Khani and Pourmahdian \cite{KP} explicitly constructed a metric on a specified cone metric space and proved that cone metric spaces are metrizable.

Recently, Jleli and Samet \cite{JS1} coined another exciting generalization of our usual metric space concept. By means of a certain class of functions, the authors defined the notion of an $\mathcal{F}$-metric space. Firstly, we recall the definition of such kind of metric spaces. Consider $\mathcal{F}$ be any set of functions $f:(0,\infty)\rightarrow \mathbb{R}$ which satisfy the following conditions:

($\mathcal{F}_1$) $f$ is non-decreasing, i.e., $ 0<s<t\Rightarrow f(s)\leq f(t)$.

($\mathcal{F}_2$) For every sequence $\{t_n\}_{n\in \mathbb{N}}\subseteq (0,+\infty)$, we have
$$\lim_{n\to +\infty}t_n=0 \Longleftrightarrow \lim_{n\to +\infty}f(t_n)=-\infty.$$

Employing these functions the authors came up with an extension of the concept of usual metric spaces and constructed the notion of $\mathcal{F}$-metric spaces as follows:
\begin{definition} \cite{JS1} \label{D1}
Let $X$ be a non-empty set and let $D:X\times X\rightarrow [0,\infty)$ be a given mapping. Suppose that there exists $(f,\alpha)\in \mathcal{F}\times [0,\infty)$ such that
\begin{enumerate}
\item[(D1)] $(x,y)\in X \times X,~~ D(x,y)=0\Longleftrightarrow x=y$.
\item[(D2)] $D(x,y)=D(y,x),$ for all $(x,y)\in X \times X$.
\item[(D3)] For every $(x,y)\in X\times X$, for each $N\in \mathbb{N}, N\geq2$, and for every $(u_i)_{i=1}^{N}\subseteq X $ with $(u_1,u_N)=(x,y)$, we have
$$D(x,y)>0 \Longrightarrow f(D(x,y))\leq f\Big(\sum_{i=1}^{N-1} D(u_i, u_{i+1})\Big)+ \alpha.$$
\end{enumerate}
Then $D$ is said to be an $\mathcal{F}$-metric on $X$, and the pair $(X,D)$ is said to be an $\mathcal{F}$-metric space.
\end{definition}
In this article, by means of the undermentioned Theorem \ref{FMS}, we assert that the recently proposed $\mathcal{F}$-metric spaces \cite{JS1} are metrizable.  Hence we also state that the new metric induces same notions of convergence, Cauchy sequence and completeness as that of the usual metric spaces. Moreover, we confirm that the celebrated Banach contraction principle in $\mathcal{F}$-metric context can be derived from its existing standard metric counterpart. Furthermore, we construct a non-trivial example of an $\mathcal{F}$-metric space with $(f,\alpha)\in \mathcal{F}\times [0,\infty)$ such that $f$ is not continuous from right at some point in $(0,\infty).$

\section{\bf Metrizability of \texorpdfstring{$\mathcal{F}$}{}-metric spaces}
In this section, we deal with the metrizability of newly introduced $\mathcal{F}$-metric spaces defined by Jleli and Samet in \cite{JS1}.
\begin{theorem} \label{FMS}
Let $X$ be an $\mathcal{F}$-metric space with $(f,\alpha)\in \mathcal{F}\times [0,\infty)$. Then $X$ is metrizable.
\end{theorem}

\begin{proof}
Let $(X,D)$ be an $\mathcal{F}$-metric space with $(f,\alpha)\in \mathcal{F}\times [0,\infty)$. We will consider two classes of functions $\mathfrak{F_1}$ and $\mathfrak{F_2}$ where $$\mathfrak{F_1}=\{f\in \mathcal{F}: f~ \mbox{is continuous in}~ (0,\infty)\}$$ and $$\mathfrak{F_2}=\{f~ \in \mathcal{F}: f~ \mbox{is not continuous from right at some point in}~ (0,\infty)\}.$$
First of all suppose that $(X,D)$ is an $\mathcal{F}$-metric space with $(f,\alpha)\in \mathfrak{F_1}\times [0,\infty)$. In \cite[\, Theorem 3.1]{JS1}, it is shown that $X$ is $\mathcal{F}$-metric bounded with respect to $(f,\alpha)\in \mathfrak{F_1}\times [0,\infty)$. So there exists a metric $d:X \times X \rightarrow [0, \infty)$ such that if $x,y \in X$ with $D(x,y)>0$,  then $$ f(d(x,y))\leq f(D(x,y))\leq f(d(x,y))+\alpha.$$

Now suppose that $\tau$ denotes the topology generated by $d$ and $\tau_F$ denotes the topology generated by $D.$

Suppose that $U \in \tau$ and $x\in U.$ Then there exists $r>0$ such that $x\in B_d (x,r)\subset U.$ If   $y\in B_D (x,r)$ and $ |B_D (x,r)|=1$,  then $y \in B_d (x,r).$ On the other hand, if $ |B_D (x,r)|>1$ then let $y \in B_D (x,r),~ y\neq x.$ Then $D(y,x)<r$ and $D(x,y)>0.$ Now since $X$ is $\mathcal{F}$-metric bounded so
\begin{align*}
&f(d(x,y))\leq f(D(x,y))\\
&\Rightarrow d(x,y)\leq D(x,y)~(\mbox{by}~\mathcal{F}_1~\mbox{condition})\\
&\Rightarrow d(x,y)<r\\
&\Rightarrow y \in B_d (x,r)\\
&\Rightarrow B_D (x,r)\subset B_d (x,r).
\end{align*}
So $x\in B_D (x,r)\subseteq B_d (x,r)\subseteq U.$ This shows that $\tau \subseteq \tau_F.$

Now let $V\in \tau_F.$ So there exists $r>0$ such that $x\in B_D (x,r)\subseteq V.$ By $\mathcal{F}_2$ condition, for$(f(r)-\alpha)$ there exists a $\delta>0$ such that $0<t<\delta$ implies $f(t)<f(r)-\alpha.$

If   $y\in B_d (x,\delta)$ and $ |B_d (x,\delta)|=1$,  then $y \in B_D (x,r).$ On the other hand, if $ |B_d (x,\delta)|>1$ then let $y \in B_d (x,\delta),~ y\neq x.$  Since $X$ is $\mathcal{F}$-metric bounded so
\begin{align*}
f(D(x,y))\leq & f(d(x,y))+\alpha\\
\Rightarrow f(D(x,y))< & f(r)-\alpha+\alpha = f(r) ~(\mbox{since}~d(x,y)<\delta)\\
\Rightarrow D(x,y)<&r\\
\Rightarrow y \in & B_D (x,r)\\
\Rightarrow B_d (x,\delta)\subseteq & B_D (x,r)
\end{align*}

So $x\in B_d (x,\delta)\subset B_D (x,r)\subseteq V.$ This shows that $\tau_F \subseteq \tau.$

Hence $\tau=\tau_F.$ In this case $X$ is metrizable.

Now suppose that $(X,D)$ is an $\mathcal{F}$-metric space with $(f,\alpha)\in \mathfrak{F_2}\times [0,\infty)$. In \cite[\, Theorem 3.1]{JS1}, the authors defined the metric $d:X\times X\rightarrow [0,\infty)$ as
$$d(x,y)=\mbox{inf}\left\{\sum_{i=1}^{N-1}D(u_i, u_{i+1}): N\in \mathbb{N}, N\geq 2, \{u_i\}_{i=1}^{N}\subseteq X~\mbox{with}~(u_1,u_N)=(x,y)\right\}.$$ They also showed that for any $\varepsilon>0, x,y\in X, y\neq x$
$$f(D(x,y))\leq f(d(x,y)+\varepsilon)+\alpha.$$

So if $f\in \mathfrak{F_2}$, then using \cite[\, Theorem 3.1]{JS1}, we have, \[ f(d(x,y))\leq f(D(x,y))\leq f(d(x,y)+\varepsilon)+\alpha\] where $\varepsilon>0, ~y\neq x.$ Now we want to show that $\tau=\tau_F .$ We can show by using the similar approach as above case that $\tau \subseteq \tau_F$.  Now we will show that $\tau_F \subseteq \tau.$

Let $V\in \tau_F$ and $x\in V.$ So there exists $r>0$ such that $x\in B_D (x,r)\subset V.$ By ($\mathcal{F}_2$) condition, for $(f(r)-\alpha)$ there exists a $\delta>0$ such that $0<t<\delta$ implies $f(t)<f(r)-\alpha.$ Choose $0<\varepsilon<\delta.$ If $y\in B_d (x,\delta - \varepsilon)$ and $ |B_d (x,\delta - \varepsilon)|=1$,  then $y \in B_D (x,r).$ On the other hand, if $|B_d (x,\delta -\varepsilon)|>1$ then for $y \in B_d (x,\delta - \varepsilon),~ y\neq x$, we have
\begin{align*}
d(x,y)&<\delta-\varepsilon\\
\Rightarrow d(x,y)+\varepsilon &<\delta\\
\Rightarrow f(d(x,y)+\varepsilon)&<f(r)-\alpha\\
\Rightarrow f(D(x,y))&<f(r)\\
\Rightarrow D(x,y)&<r.
\end{align*}
  So we have $x\in B_d (x,\delta-\varepsilon)\subset B_D (x,r)\subseteq V.$ This shows that $\tau_F \subseteq \tau.$
Hence $\tau=\tau_F.$ Therefore we can conclude that $X$ is metrizable.
\end{proof}

In \cite{JS1}, the authors presented an  examples of $\mathcal{F}$-metric spaces  with $(f,\alpha)\in \mathcal{F}\times [0,\infty)$ such that $f$ is continuous in $(0,\infty)$. But here in the subsequent example, we construct an $\mathcal{F}$-metric space with $(f,\alpha)\in \mathcal{F}\times [0,\infty)$ such that $f$ is not continuous at some point in $(0,\infty)$.

\begin{example}
Let $X=\{0\}\cup \{\frac{1}{n}: n=1,2,...,100\}$ and $D:X\times X\rightarrow [0,\infty)$ be defined by
\[D(x,y)=
\begin{cases}
0,\;~\mbox{if}~ x=y;\\
n, ~\mbox{if}~ x=0, y=\frac{1}{n}~ \mbox{or}~ x=\frac{1}{n}, y=0~\mbox{for some}~n\in \{1,2,\dots,100\};\\
2\mid m-m'\mid +100, \; ~\mbox{if}~ x=\frac{1}{m} \;  y=\frac{1}{m'}~\mbox{with}~m\neq m';\\
\end{cases}
\] for all $x,y\in X$. Consider an $f\in \mathcal{F}$ such that
\[f(t)=
\begin{cases}
-\frac{1}{t}  ,\;0<t\leq 1\\
 t   ,  \;t>1\\
\end{cases}
\]
which is not continuous from right at $t=1.$ It can be easily proved that $D$ satisfies $(D_1)$ and $(D_2)$. Let $(U_i)_{i=1}^{N}\subseteq X$, where $N\in \mathbb{N}$, $N\geq 2$ and $U_1=x, U_N=y$.
Let $I=\left\{i\in \{1,2,3,...,N-1\}: U_i\neq U_{i+1}, U_i=0, U_{i+1}=\frac{1}{n}~ \mbox{for some } n\in \{1,2,\dots,100\}\right\}$ and $J=\left\{i\in \{1,2,3,...,N-1\} :U_i\neq U_{i+1}, U_i=\frac{1}{m}, U_{i+1}=\frac{1}{m'}~ \mbox{for some } m,m'\in \{1,2,\dots, 100\}\right\}$. If $|I|=p$, then $|J|=N-1-p.$ Let $\alpha=300.$ Now we the following cases arise:

\textbf{Case I}: Let $D(x,y)>1$. For $x=0, y=\frac{1}{t},t=2,3,4,\dots 100$. Now,
\begin{align*}
f(D(x,y))-f\left(\sum_{i=1}^{N-1}D(U_i,U_{i+1})\right)-\alpha &=t-\\&f\left(\sum_{i\in I}D(U_i,U_{i+1})+\sum_{i\in J}D(U_i,U_{i+1})\right)-\alpha\\
&\leq 100-\\&f\left(\sum_{i\in I} n+\sum_{i\in J }2\mid m-m'\mid +100\right)-\alpha\\
&\leq 100-\\&\left(\sum_{i\in I} n+\sum_{i\in J }2\mid m-m'\mid +100\right)-\alpha\\
&\leq 100-p-100(N-1-p)-\alpha\\
&=200+99p-100N-\alpha\\
&\leq 200+99(N-1)-100N-\alpha\\
&=101-N-\alpha\\
&\leq 99-\alpha\\
&\leq 0.
\end{align*}
\textbf{Case II}: Let $D(x,y)=1$. For $x=0, y=1$. Now,
\begin{align*}
f(D(x,y))-f\left(\sum_{i=1}^{N-1}D(U_i,U_{i+1})\right)-\alpha &=-1-\\&f\left(\sum_{i\in I}D(U_i,U_{i+1}))+\sum_{i\in J}D(U_i,U_{i+1})\right)-\alpha\\
&= -1-\\&f\left(\sum_{i\in I} n+\sum_{i\in J }2\mid m-m'\mid +100\right)-\alpha\\
&= -1-\left(\sum_{i\in I} n+\sum_{i\in J }2\mid m-m'\mid +100\right)-\alpha\\
&\leq -1-p-100(N-1-p)-\alpha\\
&\leq 99+99(N-1)-100N-\alpha\\
&\leq -2-\alpha\\
&\leq 0.
\end{align*}
\textbf{Case III}: Let $\mid I\mid =N-1$,$\mid J \mid =0$ . For $x=0, y=\frac{1}{t},t=2,3,4,...,100$. Now,
\begin{align*}
f(D(x,y))-f\left(\sum_{i=1}^{N-1}D(U_i,U_{i+1})\right)-\alpha &=t-f\left(\sum_{i\in I}D(U_i,U_{i+1})\right)-\alpha\\
&\leq 100-f\left(\sum_{i\in I} n\right)-\alpha\\
&\leq 100-\sum_{i\in I} n-\alpha\\
&\leq 100-(N-1)-\alpha\\
&\leq 101-N-\alpha\\
&\leq 99-\alpha\\
&\leq 0.
\end{align*}
\textbf{Case IV}: Let $\mid J\mid =N-1$,$\mid I \mid =0$ . For $x=0, y=\frac{1}{t},t=2,3,4,...,100$. Now,
\begin{align*}
f(D(x,y))-f\left(\sum_{i=1}^{N-1}D(U_i,U_{i+1})\right)-\alpha &=t-f\left(\sum_{i\in J}D(U_i,U_{i+1})\right)-\alpha\\
&\leq 100-f\left(\sum_{i\in J }2\mid m-m'\mid +100\right)-\alpha\\
&\leq 100-\left(\sum_{i\in J }2\mid m-m'\mid +100\right)-\alpha\\
&\leq 100-100(N-1)-\alpha\\
&=200-100N-\alpha\\
&\leq -\alpha\\
&\leq 0.
\end{align*}
\textbf{Case V}: For $x=\frac{1}{n}, y=\frac{1}{m}$ for some $n,m\in \mathbb{N}$. Now,
\begin{align*}
f(D(x,y))-f\left(\sum_{i=1}^{N-1}D(U_i,U_{i+1})\right)-\alpha &=f(D(x,y))-\\&f\left(\sum_{i\in I}D(U_i,U_{i+1}))+\sum_{i\in J}D(U_i,U_{i+1})\right)-\alpha\\
&= f(D(x,y))-\\& f\left(\sum_{i\in I} n+\sum_{i\in J }2\mid m-m'\mid +100\right)-\alpha\\
&=2\mid n-m \mid +100-\\ & \left(\sum_{i\in I} n+\sum_{i\in J }2\mid m-m'\mid +100\right)-\alpha\\
&\leq 298-p-100(N-1-p)-\alpha\\
&=398+99p-100N-\alpha\\
&\leq 398+99(N-1)-100N-\alpha\\
&=299-N-\alpha\\
&\leq 297-\alpha\\
&\leq 0.
\end{align*}
\textbf{Case VI}: Let  $\mid I\mid =N-1$,$\mid J\mid =0$ and $x=\frac{1}{n}, y=\frac{1}{m}$ for some $n,m\in \mathbb{N}$. Now,
\begin{align*}
&f(D(x,y))-f\left(\sum_{i=1}^{N-1}D(U_i,U_{i+1})\right)-\alpha \\
&=f(D(x,y))-f\left(\sum_{i\in I}D(U_i,U_{i+1})\right)-\alpha\\
&=2\mid n-m \mid +100-\sum_{i\in I} n-\alpha\\
&\leq 298-(N-1)-\alpha\\
&=299-N-\alpha\\
&\leq 297-\alpha\\
&\leq 0.
\end{align*}
\textbf{Case VII}: Let  $\mid I\mid =0$,$\mid J\mid =N-1$ and $x=\frac{1}{n}, y=\frac{1}{m}$ for some $n,m\in \mathbb{N}$.

Now,
\begin{align*}
&f(D(x,y))-f\left(\sum_{i=1}^{N-1}D(U_i,U_{i+1})\right)-\alpha \\
&=2\mid n-m \mid +100-f\left(\sum_{i\in J}D(U_i,U_{i+1})\right)-\alpha\\
&\leq 298-100(N-1)-\alpha\\
&=398-100N-\alpha\\
&\leq 198-\alpha\\
&\leq 0.
\end{align*}
\end{example}
Combining all the cases we have $$x,y\in X, D(x,y)>0\Rightarrow f(D(x,y))\leq f\left(\sum_{i=1}^{N-1}D(U_i,U_{i+1})\right)+\alpha.$$ So $X$ is an $\mathcal{F}$-metric space with $(f,\alpha)\in \mathcal{F}\times [0,\infty)$ such that $f$ is not continuous from right at $t=1.$

\begin{theorem}
Let $X$ be an $\mathcal{F}$-metric space with $(f,\alpha)\in \mathcal{F}\times [0,\infty).$ Then
\begin{enumerate}
\item[(i)]
 $\{x_n\}_{n\in \mathbb{N}}\subseteq X$ is $\mathcal{F}$-Cauchy $\Leftrightarrow \{x_n\}_{n\in \mathbb{N}}\subseteq X$ is a Cauchy sequence with respect to the metric $d$ which is defined in \cite[\, Theorem 3.1]{JS1}.

\item[(ii)]  $\{x_n\}_{n\in \mathbb{N}}\subseteq X$ is $\mathcal{F}$-convergent to $x\in X$ $\Leftrightarrow \{x_n\}_{n\in \mathbb{N}}\subseteq X$ is convergent to $x\in X$ with respect to the metric $d$ which is defined in \cite[\, Theorem 3.1]{JS1}.

\item[(iii)] $X$ is $\mathcal{F}$-complete $\Leftrightarrow X$ is complete with respect to the metric $d$ which is defined in \cite[\, Theorem 3.1]{JS1}.

\end{enumerate}

\end{theorem}

\begin{proof}
(i) Suppose that $\{x_n\}_{n\in \mathbb{N}}\subseteq X$ is $\mathcal{F}$-Cauchy and let $\varepsilon>0.$ Then there exists $K\in \mathbb{N}$ such that $$D(x_n,x_m)<\varepsilon~\forall~ n,m\geq K$$
$$\Rightarrow d(x_n,x_m)\leq D(x_n,x_m)<\varepsilon~\forall~ n,m\geq K~(\mbox{from the definition of the metric}~d).$$
This shows that $\{x_n\}_{n\in \mathbb{N}}\subseteq X$ is a Cauchy sequence with respect to the metric $d.$

For the reverse implication, let $\{x_n\}_{n\in \mathbb{N}}\subseteq X$ be a Cauchy sequence with respect to the metric $d.$ Let $\varepsilon>0.$ By $\mathcal{F}_2$ condition, for $(f(\varepsilon)-\alpha)$ there exists a $\delta>0$ such that $0<t<\delta$ implies $f(t)<f(\varepsilon)-\alpha.$ Take $\frac{\delta}{2}>0.$ So there exists $K\in \mathbb{N}$ such that $$d(x_n,x_m)<\frac{\delta}{2}~\forall~n,m\geq K.$$

Let $n,m \geq K, x_n\neq x_m.$ Then we have the inequality
$$f(D(x_n,x_m))\leq f(d(x_n,x_m)+\frac{\delta}{2})+\alpha$$
$$\Rightarrow f(D(x_n,x_m))<f(\varepsilon)$$
$$\Rightarrow D(x_n,x_m)<\varepsilon.$$
This shows that $\{x_n\}_{n\in \mathbb{N}}\subseteq X$ is $\mathcal{F}$-Cauchy.
Proofs of (ii) and (iii) are straightforward, so omitted.
\end{proof}

\begin{theorem}
Let $X$ be an $\mathcal{F}$-metric space with $(f,\alpha)\in \mathcal{F}\times [0,\infty)$ and $g:X\rightarrow X $ be a contraction on $X$ in the setting of $\mathcal{F}$-metric space. Then $g$ is a  contraction for some suitable metric $d$ on $X.$
\end{theorem}

\begin{proof}
Let $X$ be an $\mathcal{F}$-metric space with $(f,\alpha)\in \mathcal{F}\times [0,\infty)$. From Theorem \ref{FMS} we can say $X$  is metrizable with respect to the following metric:
$$d(x,y)=\mbox{inf}\left\{\sum_{i=1}^{N-1}D(u_i, u_{i+1}): N\in \mathbb{N}, N\geq 2, \{u_i\}_{i=1}^{N}\subseteq X~\mbox{with}~(u_1,u_N)=(x,y)\right\}.$$
Let $g:X\rightarrow X $ be a contraction mapping in the setting of $\mathcal{F}$-metric space i.e., there exists $K\in (0,1)$ such that $$D(g(x),g(y))\leq K D(x,y)~\forall~x,y\in X.$$
Let $x,y\in X$ and $\{t_n\}_{n=1}^{p}\subseteq X$ be such that $t_{1}=x$ and $t_{p}=y.$ Then $\{g(t_n)\}_{n=1}^{p}=\{u_{n}\}_{n=1}^{p}\subseteq X$ is a sequence in $X$ such that $(u_{1},u_{p})=(g(x),g(y))$ and $u_{i}=g(t_i)~\forall~ i=1,2,\dots,p.$ Now
\begin{align*}
d(g(x),g(y))&\leq \sum_{i=1}^{p-1}D(u_i,u_{i+1})\\
&= \sum_{i=1}^{p-1} D(g(t_i),g(t_{i+1}))\\
&\leq K \sum_{i=1}^{p-1}D(t_i,t_{i+1})
\end{align*}
$$\Rightarrow \frac{d(g(x),g(y))}{K}\leq \sum_{i=1}^{p-1}D(t_i,t_{i+1})$$
$$\Rightarrow \frac{d(g(x),g(y))}{K}\leq d(x,y).$$
So we have $d(g(x),g(y))\leq K d(x,y)~\forall~x,y\in X.$ So $g$ is a  contraction on $X$ with respect to the metric $d.$
\end{proof}

\begin{corollary}
It is easy to conclude from Theorem 2.3 and Theorem 2.4 that the Banach contraction principle due to Jleli and Samet in \cite[\, Theorem 5.1.]{JS1} is the direct consequence of the original Banach contraction principle for the corresponding metric space $(X,d).$
\end{corollary}

\begin{Acknowledgement}
 The Research is funded by the Council of Scientific and Industrial Research (CSIR), Government of India under the Grant Number: $25(0285)/18/EMR-II$.
\end{Acknowledgement}

\bibliographystyle{plain}

\end{document}